\documentclass[12pt]{article}
\usepackage{times,amsmath,amsthm,amssymb,graphicx,xspace,epsfig,xcolor}
\usepackage[plain]{fullpage}

\usepackage[english]{babel}
\usepackage{algorithm}\usepackage{algorithmic}
\usepackage{tikz}
\usepackage{color,hyperref}
\usepackage{comment}
\usepackage{authblk}
\usepackage{subcaption}

\usepackage{tikz, tkz-graph, tkz-berge}
\usetikzlibrary{calc}

\newtheorem{theorem}{Theorem}[section]
\newtheorem{corollary}[theorem]{Corollary}
\newtheorem{proposition}[theorem]{Proposition}
\newtheorem{lemma}[theorem]{Lemma}

\theoremstyle{definition}

\newtheorem{problem}[theorem]{Problem}
\newtheorem{conjecture}[theorem]{Conjecture}

\DeclareMathOperator{\Ad}{{\rm Ad}}

\newcommand{\dic}{\vec{\chi}}
\newcommand{\bid}{\overleftrightarrow}

\newcommand{\eq}{=}

\renewcommand{\leq}{\leqslant}
\renewcommand{\geq}{\geqslant}

\title{On the dichromatic number of surfaces}
\author[1]{Pierre Aboulker}
\author[2]{Fr\'ed\'eric Havet}
\author[3]{Kolja Knauer}
\author[1]{Cl\'ement Rambaud}

\affil[1]{DIENS, \'Ecole normale sup\'erieure, CNRS, PSL University, Paris, France}
\affil[2]{CNRS, Universit\'e C\^ote d'Azur, I3S, INRIA, Sophia Antipolis, France}
\affil[3]{Aix Marseille Univ, Universit\'e de Toulon, CNRS, LIS, Marseille, France\\ Departament de Matem\`atiques i Inform\`atica,
Universitat de Barcelona, Spain}
\date{}

\begin{document}

\maketitle

\begin{abstract}
In this paper, we give bounds on the dichromatic number  $\vec{\chi}(\Sigma)$ of a surface $\Sigma$, which is the maximum dichromatic number of an oriented graph embeddable on $\Sigma$.
We determine the asymptotic behaviour of $\vec{\chi}(\Sigma)$ by showing that there exist constants $a_1$ and $a_2$ such that, 
 $ a_1\frac{\sqrt{-c}}{\log(-c)}   \leq \vec{\chi}(\Sigma) \leq  a_2  \frac{\sqrt{-c}}{\log(-c)} $ for every surface $\Sigma$ with Euler characteristic $c\leq -2$. 
We then give more explicit bounds for some surfaces with high Euler characteristic. In particular, we show that the dichromatic numbers of the projective plane $\mathbb{N}_1$, the Klein bottle $\mathbb{N}_2$, the torus $\mathbb{S}_1$, and Dyck's surface $\mathbb{N}_3$ are all equal to $3$, and that
the dichromatic numbers of the $5$-torus $\mathbb{S}_5$ and the $10$-cross surface $\mathbb{N}_{10}$ are equal to $4$. 
We also consider the complexity of deciding whether a given digraph or oriented graph embeddable on a fixed surface is $k$-dicolourable.
In particular, we show that for any fixed surface, deciding whether a digraph embeddable on this surface is $2$-dicolourable is NP-complete, and that deciding whether a planar oriented graph is $2$-dicolourable is NP-complete unless all planar oriented graphs are $2$-dicolourable (which was conjectured by Neumann-Lara).
\end{abstract}

\section{Introduction}

All surfaces considered in this paper are closed. 
 
A graph is {\bf embeddable} on a surface $\Sigma$ if its vertices can be
mapped onto distinct points of $\Sigma$ and its edges onto simple curves of
$\Sigma$ joining the points onto which its endvertices are mapped, so that two
edge curves do not intersect except in their common extremity.  A {\bf face} of an
embedding $\tilde{G}$ of a graph $G$ is a component of $\Sigma \setminus
\tilde{G}$.
Recall that an important theorem of the topology of surfaces, known as the
Classification Theorem for Surfaces, states that every surface is
homeomorphic to either the $k$-torus -- a sphere with $k$-handles $\mathbb{S}_k$ or the $k$-cross surface -- a sphere
with $k$-cross-caps $\mathbb{N}_k$. The surface $\mathbb{S}_0=\mathbb{N}_0$ is the sphere, and the surfaces $\mathbb{S}_1$,
$\mathbb{S}_2$, $\mathbb{N}_1$, $\mathbb{N}_2$ , $\mathbb{N}_3$ are also called the
\emph{torus}, the \emph{double torus}, the \emph{projective plane}, the
\emph{Klein bottle}, and \emph{Dyck's surface}, respectively.
The {\bf Euler characteristic} of a surface homeomorphic to $\mathbb{S}_k$ is $2 -2k$ and of a surface homeomorphic to $\mathbb{N}_k$ it is $2 -k$.
We denote the Euler characteristic of a surface $\Sigma$ by $c(\Sigma)$.

Let $G$ be a graph. We denote by $n(G)$ its number of vertices, and by $m(G)$ its number of edges. If $G$ is
embedded in a surface $\Sigma$, then we denote by $f(G)$ the number of faces of the embedding.
Euler's Formula relates the numbers of vertices, edges and faces of a (connected) graph embedded in a surface.

\begin{theorem}\label{thm:euler-form}~~~{\sc Euler's Formula}\\
  Let $G$ be a connected graph embedded on a surface $\Sigma$. Then
  \[
    n(G) - m(G) + f(G) \geq c(\Sigma).
  \]
\end{theorem}

We denote by $\Ad(G) = 2m/n$ the average degree of a graph $G$. 
Euler's formula implies that graphs on surfaces have bounded average degree. 
\begin{theorem}\label{thm:euler-gen}
  A graph $G$ embeddable on a surface $\Sigma$ satisfies: 
  $$m(G) \leq 3n(G) - 3 c(\Sigma)  \mbox{ ~~~~~and~~~~~} \Ad(G) \leq 6 - \frac{6c(\Sigma)}{n(G)}.$$
  Moreover, there is equality if and only if $G$ is a triangulation.
\end{theorem}

A {\bf $k$-colouring} of a graph $G$ is a partition of the vertex set of $G$ into $k$ disjoint {\bf stable sets} (i.e. sets of pairwise non-adjacent vertices).  
A graph is {\bf $k$-colourable} if it has a $k$-colouring.  
The  {\bf chromatic number} of a graph $G$, denoted by $\chi(G)$, is the least integer $k$ such that $G$ is $k$-colourable, and the {\bf chromatic number} of a surface $\Sigma$, denoted by $\chi(\Sigma)$, is the least integer $k$ such that every graph embeddable on $\Sigma$ is $k$-colourable. 
Determining the chromatic number of surfaces attracted lots of attention, with its most important instance being the Four Colour Conjecture on planar graphs (i.e. graphs embeddable on $\mathbb{S}_0$). This conjecture was eventually proved by Appel and Haken~\cite{ApHa1977} using computer assistance and another proof requiring less  computer assistance was given by Robertson et al.~\cite{RSST96}. 
Maybe surprisingly, the chromatic numbers of surfaces other than the plane were established  before the Four Colour Theorem.  In 1890, Heawood~\cite{He1890} proved the following theorem as a consequence of Euler's Formula.

\begin{theorem}[Heawood~\cite{He1890}]\label{thm_heawood}
If $\Sigma$ is a surface with Euler characteristic $c \leq 0$, then $\chi(\Sigma) \leq H(c) = \left\lfloor \frac{7 + \sqrt{49-24c}}{2} \right\rfloor$.
\end{theorem}
 
 Franklin~\cite{Fra34} showed that the above inequality is not tight as the Klein
bottle has chromatic number $6$ (the above inequality yields $7$ as
the Klein bottle has Euler characteristic $0$). 
Contrary to the sphere, on other surfaces the most effort of determining the chromatic number went into proving the lower bounds. Indeed, Ringel and Youngs~\cite{Ri1968} proved that the Klein bottle is the sole surface that does not admit an embedding of a complete graph witnessing the Heawood bound.

 \begin{theorem}[Ringel and Youngs~\cite{Ri1968}]\label{thm:plongement-complet} 
 Let $\Sigma$ be a surface different from the Klein bottle $\mathbb{N}_2$ and let $c$ be its Euler characteristic.
 Then the complete graph of order $H(c)$ is embeddable on $\Sigma$.
\end{theorem}

The {\bf girth} of a graph is the length of a shortest cycle in it (or $+\infty$ if it is acyclic).
The chromatic number of graphs embeddable on a surface and of girth at least $g$ has been studied.
In the same way as Theorem~\ref{thm:euler-gen}, one can derive from Euler's Formula that the average degree of 
a graph $G$ embeddable on a surface $\Sigma$ and of girth at least $g$ is at most
$\frac{2g}{g-2} - \frac{2gc(\Sigma)}{n(G)}$. 
This implies that the maximum chromatic number over all graphs of girth at least $g$ embeddable on $\Sigma$ tends to $3$ when $g$ tends to $+\infty$.
A particular interest has been devoted to {\bf triangle-free} graphs, i.e., graphs of girth at least $4$.
The above bound on the average degree implies that triangle-free planar graphs have average degree at most $3$, and so are 
$4$-colourable.  The celebrated Gr\"otzsch's Theorem~\cite{Gro59} asserts
that such graphs are even $3$-colourable. A short proof can be found in \cite{Tho94}.
\begin{theorem}[Gr\"otzsch~\cite{Gro59}]\label{grotzsch}
Every triangle-free planar graph is $3$-colourable.
\end{theorem}
Kronk and White~\cite{KrWh72} proved that every triangle-free graph embeddable on the torus is $4$-colourable, and  Kronk~\cite{Kro72} studied the chromatic number of triangle-free graphs on certain surfaces. 
Asymptotic bounds on the maximum chromatic number of triangle-free graphs embeddable on a given surface have been given by Gimbel and Thomassen~\cite{GiTh97}. Here we only give the results for orientable surfaces. 
\begin{theorem}[Gimbel and Thomassen~\cite{GiTh97}]
There exist positive constants $c_1$ and $c_2$ such that the following hold:
\begin{itemize}
\item[(i)] Every triangle-free graph embeddable on $\mathbb{S}_k$ has chromatic number at most $c_1 \sqrt[3\,]{k/\log k}$.
\item[(ii)] for each $k$, there exists a triangle-free graph which is embeddable on $\mathbb{S}_k$ and with chromatic number at least  $c_2 \sqrt[3\,]{k}/\log k$.
\end{itemize}
\end{theorem}

\bigskip

In 1982, Neumann Lara~\cite{Neu1982} introduced the notion of directed colouring  or dicolouring. 
A {\bf $k$-dicolouring} of a digraph is a partition of its vertex set into $k$ subsets inducing acyclic subdigraphs.
A digraph is {\bf $k$-dicolourable} if it has a $k$-dicolouring.  
The {\bf dichromatic number} of a digraph $D$, denoted by $\vec{\chi}(D)$, is the least integer $k$ such that $D$ is $k$-dicolourable. 

Let $G$ be an undirected graph. The {\bf bidirected graph} $\bid{G}$ is the digraph obtained from $G$ by replacing each edge by a {\bf digon}, that is a pair of oppositely directed arcs between the same end-vertices. Observe that $\chi(G) = \vec{\chi}(\bid{G})$ since any two adjacent vertices in $\bid{G}$ induce a directed cycle of length $2$.

The {\bf digirth} of a digraph is the length of a smallest directed cycle in it (or $+\infty$ if it is acyclic).
In view of the influence of the girth on the chromatic number of graphs on surfaces, it is natural to study the influence of the digirth on the dichromatic number. In particular, it is interesting to study the dichromatic number of  digraphs of digirth $3$, which are called {\bf oriented graphs}.
Alternatively, oriented graphs may be seen as the digraphs which can be obtained from (simple) graphs by orienting every edge, that is replacing each edge by exactly one of the two possible arcs between its end-vertices. If $\vec{G}$ is obtained from $G$ by orienting its edges, we say that $G$ is the {\bf underlying graph} of $\vec{G}$.   
It is easy to show that oriented planar graphs are $3$-dicolourable and Neumann Lara~\cite{Neu1982} proposed the following conjecture, which can be viewed as an analogue of Gr\"otzsch's Theorem (Theorem~\ref{grotzsch}).
\begin{conjecture}[Neumann Lara~\cite{Neu1982}]\label{conj_neumann}
Every oriented planar graph is $2$-dicolourable.
\end{conjecture}

The conjecture is part of an active field of research. It has been verified for planar oriented graphs on at most 26 vertices~\cite{KV19} and holds for planar digraphs of digirth at least $4$~\cite{LM17}.

Analogously to the history of the chromatic number, in the present paper we study the dichromatic number of surfaces. 
The  {\bf dichromatic number} of a surface $\Sigma$, denoted by $\vec{\chi}(\Sigma)$, is the least integer $k$ such that every oriented graph embeddable on $\Sigma$ is $k$-dicolourable.

The {\bf arboricity} of a graph $G$, denoted by $a(G)$,  is the minimum integer $k$ such that $V(G)$ can be partitioned into $k$ sets, each of which induces a forest, that is an acyclic (non-directed) graph.

Let $\vec{G}$ be an oriented graph  and $G$ its underlying graph. Then 
 \begin{equation}\label{ineg}
 \vec{\chi}(\vec{G})\leq a(G) \leq \chi(G)
 \end{equation}

Kronk~\cite{Kr1969} proved that if a graph is embeddable on a surface $\Sigma$ with Euler characteristic $c \leq 1$, then
$a(G) \leq \left\lfloor \frac{9+\sqrt{49-24c(\Sigma)}}{4} \right\rfloor$.
By Eq. \eqref{ineg}, this immediately implies

\[\vec{\chi}(\Sigma) \leq \left\lfloor \frac{9+\sqrt{49-24c(\Sigma)}}{4} \right\rfloor\]

We first improve on this bound by determining the asymptotic behaviour of $\vec{\chi}(\Sigma)$. We show in Theorem~\ref{thm:bnd-gen}, that there exists two constants $a_1$ and $a_2$ such that, for every surface $\Sigma$ with Euler characteristic $c\leq -2$, we have
  \[
   a_1\frac{\sqrt{-c}}{\log(-c)} 
  \leq \vec{\chi}(\Sigma) \leq
  a_2  \frac{\sqrt{-c}}{\log(-c)}.
  \]

We then estimate the exact value of the dichromatic number of  surfaces close to the sphere.
Table~\ref{tab_dichi} summarizes the main results.

\begin{table}[!hbtp]
    \begin{center}
    \begin{tabular}{|c|c|c|c|}
        \hline
        $\Sigma$ & $c(\Sigma)$ & Bounds for $\vec{\chi}(\Sigma)$ & Reference \\
        \hline\hline
        Sphere $\mathbb{N}_0 = \mathbb{S}_0$ &
          $2$ & $2 \leq \vec{\chi} \leq 3 $  & Neumann Lara~\cite{Neu1982} \\
        \hline
        Projective plane $ \mathbb{N}_1$ & 
          $1$ & $\vec{\chi} = 3 $  & Corollary~\ref{cor:lowgenus} \\
        \hline
        Klein bottle $ \mathbb{N}_2$ &
          $0$ & $\vec{\chi} = 3$  & Corollary~\ref{cor:lowgenus}\\
        \hline
        Torus $ \mathbb{S}_1$ &
          $0$ & $\vec{\chi} = 3$  & Corollary~\ref{cor:lowgenus}\\
        \hline
        Dyck's surface $\mathbb{N}_3$ & $-1$ & $\vec{\chi} = 3$ & Corollary~\ref{cor:lowgenus}\\
        \hline
        $\mathbb{S}_2$, $\mathbb{N}_4$ & $-2$ & $3 \leq \vec{\chi} \leq 4$  & Theorems~\ref{thm:N1} and~\ref{thm:S5} \\
        \hline
        $\mathbb{N}_5$ & $-3$ & $3 \leq \vec{\chi} \leq 4$ & Theorems~\ref{thm:N1} and~\ref{thm:S5} \\
        \hline
        $\mathbb{S}_3$, $\mathbb{N}_6$ & $-4$ & $3 \leq \vec{\chi} \leq 4$ & Theorems~\ref{thm:N1} and~\ref{thm:S5} \\
        \hline
        $\mathbb{N}_7$ & $-5$ & $3 \leq \vec{\chi} \leq 4$ & Theorems~\ref{thm:N1} and~\ref{thm:S5} \\
        \hline
        $\mathbb{S}_4$, $\mathbb{N}_8$ & $-6$ & $3 \leq \vec{\chi} \leq 4$ & Theorems~\ref{thm:N1} and~\ref{thm:S5} \\
        \hline
        $\mathbb{N}_9$ & $-7$ & $3 \leq \vec{\chi} \leq 4$ & Theorems~\ref{thm:N1} and~\ref{thm:S5} \\
        \hline
        $\mathbb{S}_5$, $\mathbb{N}_{10}$ & $-8$ & $\vec{\chi} = 4$ & Corollary~\ref{cor:N10S5} \\
        \hline
    \end{tabular}
    \caption{\label{tab_dichi} Bounds on the dichromatic number of some surfaces.}
    \end{center}
\end{table}

 
 \bigskip

Finally, we consider the computational complexity of the related (di)colourability problems.
Regarding undirected graphs, for any surface $\Sigma$ and any integer $k\geq 5$, there are only finitely many  $(k+1)$-critical graphs  (i.e. graphs $G$ such that $\chi(G)=k+1$ and $\chi(H) \leq k$ for any proper subgraph $H$ of $G$) embeddable on $\Sigma$. This was observed by 
Dirac~\cite{Dir57} for $k\geq 6$ and proved by Thomassen~\cite{Tho97} for $k=5$.
It follows that, for any surface $\Sigma$ and any integer $k\geq 5$, there is a polynomial-time algorithm that decides whether a graph $G$ embeddable on $\Sigma$ is $k$-colourable.
For smaller values of $k$, i.e. $k\in \{2,3,4\}$, there are infinitely many $(k+1)$-critical graphs embeddable on any surface $\Sigma$ distinct from the sphere. (For $k=4$, this follows from a result of Fisk~\cite{Fisk78}.)
For $k=2$, it is polynomial-time solvable to decide whether a graph is $2$-colourable.
In contrast, deciding whether a graph embeddable on the sphere (and thus on any other surface) is  $3$-colourable is NP-complete. (See \cite{GaJo79}).
For $k=4$, the problem remains open, except for the sphere, for which there is a trivial algorithm by the Four Colour Theorem.


\medskip
We are interested in the analogous problems for dicolouring.

\smallskip

\noindent\textsc{$\Sigma$-$k$-Dicolourability}\\
\textbf{Input}: A digraph $D$ embeddable on $\Sigma$.\\
\textbf{Question}: Is $D$ $k$-dicolourable ?

\smallskip

A natural idea is to consider $(k+1)$-dicritical digraphs. A digraph $D$ is {\bf $(k+1)$-dicritical} if $\vec{\chi}(D)=k+1$ and $\vec{\chi}(H)\leq k$ for every proper subdigraph $H$ of $D$. 
One easily derives from Euler's Formula that,  for any $k\geq 7$,  the number of $(k+1)$-dicritical digraphs embeddable on a surface is finite (Proposition~\ref{prop_8crit}). Adapting the standard method for critical graphs to dicritical digraphs, we prove in Corollary~\ref{cor:finite-dicrit}, that  the number of $(k+1)$-dicritical digraphs embeddable on a surface is finite for any $k\geq 6$.
Consequently, for any surface $\Sigma$ and any integer $k\geq 6$, one can solve \textsc{$\Sigma$-$k$-Dicolourability} in polynomial time.

When $k=2$, in contrast to the undirected case, Bokal et al.~\cite{Bokal2004} showed that deciding whether a digraph is $2$-dicolourable is NP-complete. We show (Theorem~\ref{thm:planar-2-NP}) that it remains NP-complete when restricted to digraphs embeddable on the sphere (and hence in any surface). In other words, \textsc{$\Sigma$-$2$-Dicolourability} is NP-complete for any surface $\Sigma$.
Since the chromatic number of $G$ is equal to the dichromatic number of the bidirected graph $\bid{G}$, the NP-completeness of the $3$-colourability of a graph embeddable in $\Sigma$ implies that \textsc{$\Sigma$-$3$-Dicolourability} is NP-complete for any surface $\Sigma$.
The complexity of \textsc{$\Sigma$-$k$-Dicolourability} for $k\in \{4,5\}$ and $\Sigma$ different from the sphere remains open, see Problem~\ref{prob:45}.


\medskip

We then consider the restriction of \textsc{$\Sigma$-$k$-Dicolourability} to oriented graphs.

\smallskip

\noindent \textsc{$\Sigma$-Oriented-$k$-Dicolourability}\\
\textbf{Input}: An oriented $D$ embeddable on $\Sigma$.\\
\textbf{Question}: Is $D$ $k$-dicolourable ?

\smallskip

For any surface $\Sigma$ and any integer $k\geq 3$, there are only finitely many  $(k+1)$-dicritical oriented graphs embeddable on $\Sigma$. For $k\geq 4$, this follows easily from Euler's Formula (see Proposition~\ref{prop_5crit}); for $k=3$, it was proved by  Kostochka and Stiebitz~\cite{KosStie2020} (See Theorem~\ref{thm:4crit}). This implies that, for any surface $\Sigma$ and any integer $k\geq 3$, one can solve \textsc{$\Sigma$-Oriented-$k$-Dicolourability} in polynomial time.
Hence we are left with the complexity of \textsc{$\Sigma$-Oriented-$2$-Dicolourability}. 


If Conjecture~\ref{conj_neumann} is true, then {\sc $\mathbb{S}_0$-ORIENTED-2-Dicolourability} can be trivially solved in polynomial time because the answer is always positive. Conversely, we show in Theorem~\ref{thm:NP} that if  Conjecture~\ref{conj_neumann}  is false then {\sc $\mathbb{S}_0$-ORIENTED-2-Dicolourability} is NP-complete. 

\section{Preliminaries}

\subsection{Dicritical oriented graphs}


Recall that a digraph $D$ is {\bf $k$-dicritical} if $\vec{\chi}(D)=k$ and $\vec{\chi}(H)<k$ for every proper subdigraph $H$ of $D$.
In this subsection, we establish some results on dicritical oriented graphs 
which will be useful to prove our main results.

The following proposition is well-known and easy to prove. Note that it is the only result of this section about digraphs, every other result is about oriented graphs.
\begin{proposition}\label{prop_crit_deg_digraphe}
Let $D$ be a  $k$-dicritical digraph.
Then $d^+(v),d^-(v) \geq k-1$ for all $v\in V(D)$.
\end{proposition}

The next  result gives an upper bound on the number of $k$-dicritical oriented graphs embeddable on a surface, when $k \geq 5$. 
This result is used to prove Theorem~\ref{thm:bnd-gen}. 
\begin{proposition}\label{prop_5crit}
Let $k\geq 5$ and let $\vec{G}$ be a  $k$-dicritical oriented graph embedded in a surface with Euler characteristic $c$. Then
  $$
  n(\vec{G}) \leq \frac{-3c}{k-4}
  $$
\end{proposition}

\begin{proof}
By Proposition~\ref{prop_crit_deg_digraphe}, $d^+(v),d^-(v) \geq k-1$ for every vertex $v$ of $\vec{G}$.
Moreover, by Theorem~\ref{thm:euler-gen}, $\Ad(\vec{G}) \leq 6 -  \frac{6c}{n(\vec{G})}$
and so $2(k-1) \leq 6 - \frac{6c}{n(\vec{G})}$.
Hence $n(\vec{G}) (k-1-3) \leq -3c$ which yields the result.
\end{proof}

Kostochka and Stieblitz~\cite{KosStie2020} proved that for every $4$-dicritical oriented graph $G$,  $3m(\vec{G}) \geq 10 n(\vec{G}) - 4$. It implies that the $4$-dicritical oriented graphs that are embeddable on a given surface have bounded number of vertices. This result is used to prove Theorem~\ref{thm:bnd-gen} and Theorem~\ref{thm:N3}.  


 
\begin{theorem}[Kostochka and Stiebitz~\cite{KosStie2020}]\label{thm:4crit} 
 Let $\vec{G}$ be a $4$-dicritical oriented graphs embedded in a surface with Euler characteristic $c$. Then $n(\vec{G}) \leq 4 - 9c$.
\end{theorem}
\begin{proof}
By Euler's formula, we have $m(\vec{G}) \leq 3n(\vec{G}) - 3c$ and thus
$10n(\vec{G}) -4 \leq 3 m(\vec{G}) \leq 9n(G) - 9c$. Therefore $n(\vec{G}) \leq 4 - 9c$.
\end{proof}

A graph $G$ is {\bf non-separable} if it is connected and $G-v$ is connected for all $v\in V(G)$.
Let $G$ be a graph. A {\bf block} of $G$ is a subgraph which is non-separable and is maximal with respect to this property. 
A {\bf cactus}  is a graph whose blocks are cycles or edges.
 A {\bf directed cactus} is an oriented graph whose blocks are directed cycles or arcs.
 In other words, a directed cactus is an oriented cactus in which every cycle is directed.
The following result is the main tool in the proofs of 
 Theorem~\ref{thm:N3} and Theorem~\ref{thm:S5}. We state it here for oriented graphs, but a more general version holding for digraphs exists and is used latter in the paper, see~Theorem~\ref{thm:gallai}. 

\begin{theorem}[Bang-Jensen et al.~\cite{hajosconstruct2019}] \label{thm:gallai_or}
Let $\vec{G}$ be a $k$-dicritical oriented graph.
The subdigraph induced by the vertices of in- and out-degree $k-1$  is a directed cactus.
\end{theorem}

The next two results are technical lemmas on cacti that we use to prove Theorem~\ref{thm:N3} and Theorem~\ref{thm:S5}. 

\begin{lemma}\label{fait_t}
Let $G$ be a cactus.
Then $m(G) \leq \frac{3}{2}(n(G) - 1)$.
Moreover,  equality holds if and only if $G$ is connected and every block is a triangle.
\end{lemma}
\begin{proof}
We proceed by induction on $n(G)$ the number of vertices of $G$, the result being trivial if $n(G)\leq 2$.

Suppose now that $n(G) > 2$. 
If $G$ is not connected, then applying the induction hypothesis on each connected component, and summing the obtained inequalities give the result. 

Suppose now that $G$ is connected. If it is a cycle or an edge, then $m(G)\leq n(G)$ so the result holds. 
If $G$ is not a cycle, then $G$ contains a leaf block $C$ with attachment $x$. 
Let $G'$ be the graph obtained from $G$ by deleting all vertices of $C$ except $x$.
We have $n(G')= n(G) - n(C) +1$ and $m(G') = m(G) - m(C)$ and $m(C) = n(C)  - \epsilon(C)$, with $\epsilon(C) =1$ if and only if $C$ is an edge, and $\epsilon(C)=0$ otherwise. 
By the induction hypothesis, $m(G') \leq \frac{3}{2} (n(G') - 1)$.
Hence
\[
\begin{split}
m(G) &\leq \frac{3}{2} (n(G) - n(C)+1 - 1) + n(C) - \epsilon(C) \\
     &=    \frac{3}{2} (n(G) - 1) - \frac{1}{2} n(C) + \frac{3}{2} -\epsilon(C) \\
     &\leq \frac{3}{2} (n(G) - 1)
\end{split}
\]
Moreover, $- \frac{1}{2} n(C) + \frac{3}{2} -\epsilon(C) = 0$ if and only if $C$ is a triangle, which implies that the bound is tight if and only if all blocks are triangles. 
\end{proof}

\begin{lemma}\label{lem:foret-cactus}
Every cactus of order $n$ contains an induced forest of order $\lceil \frac{2}{3}n \rceil$.
\end{lemma}
\begin{proof}
By induction on $n$.
Let $G$ be a cactus of order $n$.
If $G$ is not connected, then we have the result by applying the induction on each connected component and summing the obtained inequalities.

If $G$ is a cycle or an edge, the result is clear.
If not, then $G$ admits a leaf block $C$ with attachment $x$. 
Let $G'$ be the graph obtained from $G$ by removing all vertices of $V(C)$ except $x$. 
By the induction hypothesis, $G'$ contains a set $X$ of order at least 
$\lceil \frac{2}{3}(n-n(C)+1) \rceil$ such that the induced subgraph $G\langle X\rangle$ is a forest.
Let $y$ be a vertex of $V(C)\setminus \{x\}$.
If $C$ is an edge, then $G\langle X\cup \{y\}\rangle $ is a forest of order  at least  $\lceil \frac{2}{3}(n-n(C)+1) \rceil +1 \geq \lceil \frac{2}{3}n \rceil$. 
If $C$ is a cycle, then $G\langle X\cup V(C) \setminus \{y\}\rangle $ is a forest of order  at least $\lceil \frac{2}{3}(n-n(C)+1) \rceil + n(C) -1 \geq  \lceil \frac{2}{3}n \rceil$ because $n(C)\geq 3$.
\end{proof}

%
%
%
%
%


\subsection{Large acyclic subdigraphs and dichromatic number in oriented graphs}

Note that, in a tournament, the induced acyclic subdigraphs are the transitive subtournaments. The following classic result of Stearns is used to prove Proposition~\ref{prop_n_leq15}. 

\begin{lemma}[Stearns~\cite{S59}]\label{lem_acyclique_log}
Every tournament of order $n$ has an induced acyclic subdigraph of order $\lfloor \log_2 n \rfloor + 1$. 
\end{lemma}

Erd\H{o}s and Moser~\cite{erdos_moser} proved that, for every $n \geq 2$, there exists a tournament of order $n$   whose largest acyclic subdigraph has at most  $\lfloor 2\log_2 n \rfloor + 1$ vertices.
Since the dichromatic number of an $n$-vertex digraph is at least $n$ divided by the order of a largest acyclic induced subdigraph, we get the following, used to prove Theorem~\ref{thm:bnd-gen}. 

\begin{proposition}[Erd\H{o}s and Moser~\cite{erdos_moser}]\label{prop:exist-tournament}
 For every $n \geq 2$, there exists a tournament $T$ of order $n$ such that $\vec{\chi}(T) \geq \frac{n}{2\log (n) + 1}$.
\end{proposition}



We now turn our attention to acyclic subdigraphs in small digraphs.
A digraph is \textbf{$k$-diregular} if all its vertices have in- and out-degree $k$. The next two results are used to prove Theorem~\ref{thm:S5}. 

\begin{theorem}[Reid and Parker~\cite{tournoi_14} ; Sanchez-Flores~\cite{sanchez1998tournaments}]\label{thm_tournament_14}~ 
\begin{itemize}
\item[(i)] Every tournament of order $14$ has a transitive subtournament of order $5$.

\item[(ii)] There is a unique tournament $ST_{13}$ of order $13$ with no transitive subtournament of order $5$. 
This tournament $ST_{13}$ is $6$-diregular.

\item[(iii)] There is a unique tournament $ST_{12}$ of order $12$ with no transitive subtournament of order $5$. 
This tournament is obtained from $ST_{13}$ by removing any vertex.
\end{itemize}
\end{theorem}

\begin{corollary}\label{cor_tournament_13_12}~
\begin{itemize}
    \item[(i)] The only oriented graph of order $13$ with no acyclic induced subdigraph of order $5$ is $ST_{13}$.    
    \item[(ii)] Every oriented graph of order $12$ with no acyclic induced subdigraph of order $5$ satisfies
    $d^+(v),d^-(v) \geq 5$ for every vertex $v$. In particular it has at least $60$ arcs.
\end{itemize}
\end{corollary}
\begin{proof}
(i) Let $\vec{G}$ be an oriented graph of order $13$ distinct from $ST_{13}$.
If $\vec{G}$ is a tournament, then we have the result by Theorem~\ref{thm_tournament_14}~(ii).
Assume now that $\vec{G}$ is not a tournament.Then there exist non-adjacent vertices $x,y$ in $\vec{G}$.
We add arcs to $\vec{G}$ to obtain a tournament $T$ in which $x$ does not have in-degree $6$.
By Theorem~\ref{thm_tournament_14}~(ii), $T$ has a transitive subtournament of order $5$, and thus
$\vec{G}$ has an acyclic induced subdigraph of order $5$.

\medskip

(ii) Let $\vec{G}$ be an oriented graph of order $12$. 
Assume that there is a vertex $v$ with out-degree at most $4$ in $\vec{G}$. 
We add arcs to $\vec{G}$ to obtain a tournament  $T$ in which $x$ has the same out-degree as in $\vec{G}$.
Now $T\neq ST_{12}$ because every of $ST_{12}$ has out-degree $5$ and $6$.
Hence, by Theorem~\ref{thm_tournament_14}~(ii), $T$ has a transitive subtournament of order $5$, and thus
$\vec{G}$ has an acyclic induced subdigraph of order~$5$.

Similarly, if  there is a vertex $v$ with out-degree at most $4$, then $\vec{G}$ has an acyclic induced subdigraph of order $5$.
\end{proof}

\subsection{Small oriented graphs of dichromatic number $3$ or $4$}

We describe the $3$- and $4$-dichromatic oriented graphs with a given number of vertices. These results are used in the proof of Theorem~\ref{thm:N1}, Theorem~\ref{thm:N3} and Theorem~\ref{thm:S5}. 

\begin{proposition}\label{prop_3crit_edges}

    \begin{itemize}
    \item[(i)] All oriented graphs on at most $6$ vertices are $2$-dicolourable.
        \item[(ii)] The unique smallest $3$-dicritical oriented graph on $7$ vertices has $20$ arcs. 
        \item[(iii)] The unique smallest $3$-dicritical oriented graph on $8$ vertices has $21$ arcs.
        \item[(iv)] The unique smallest $3$-dicritical oriented graph on $9$ vertices has $23$ arcs.
        \item [(v)] Every $3$-dicritical oriented graph on at least $10$ vertices has at least $21$ arcs.
    \end{itemize}
\end{proposition}
\begin{proof}
(i) was shown in~\cite{NL94}. 

We verified (ii), (iii), and (iv) by exhaustive computation. By Proposition~\ref{prop_crit_deg_digraphe}, $3$-dicritical  oriented graphs have minimum in- and out-degree at least $2$ and thus their underlying graph have minimum degree $4$. 
We used the program \textit{nauty}~\cite{McKay201494} to generate all
 (non-oriented) graphs of order $7$, $8$ and $9$ of minimum degree at least $4$.
From this list, we extracted the graphs of arboricity at least $3$
and generated the orientations of minimum in- and out-degree at least $2$
using McKay's program \textit{nauty}~\cite{McKay201494}.
Finally, we kept only the orientations that were $3$-dicritical.
The code can be found at~\url{https://github.com/ClementRambaud/cdicoloring}.


To prove \textit{(v)}, observe that $3$-dicritical oriented graphs have in- and out-degree at least $2$ (Proposition~\ref{prop_crit_deg_digraphe}) and  that Theorem~\ref{thm:gallai_or} implies that a $3$-dicritical oriented graph cannot be $2$-diregular. 
\end{proof}

%
%


\begin{theorem}[Neumann Lara~\cite{NL94}]\label{thm_nl_clique}
All oriented graphs on at most $10$ vertices are $3$-dicolourable. The unique smallest oriented graph with dichromatic number $4$ on $11$ vertices has $55$ edges. It is depicted in Figure~\ref{fig_T11}.
\end{theorem}
\begin{proof}
In~\cite{NL94}, Neumann-Lara proved that every tournament (and so oriented graph) of order $10$ is $3$-dicolourable, and that there is a unique tournament $ST_{11}$ with order $11$ and dichromatic number $4$.
It is depicted in Figure~\ref{fig_T11}. Let us show that this tournament is dicritical, which implies the result. 

Since $ST_{11}$ is arc-transitive (for any two arcs $e$ and $e'$ there is an automorphism sending $e$ onto $e'$), it suffices to show an arc $e$ such that $T\setminus e$ is $3$-dicolourable.
Let us use the vertex numbering of Figure~\ref{fig_T11}.
Observe that in $T\setminus (4,2)$, the set $S=\{0,1,2,4,5\}$ induces an acyclic subdigraph of order $5$.
Now $ST_{11}-S$ has $6$ vertices and is $2$-dicolourable by \textit{(i)}. Hence $T\setminus (4,2)$ is $3$-dicolourable and $T$ is $4$-dicritical.
\end{proof}

\begin{figure}[ht]
\centering
\begin{tikzpicture}[scale=5, rotate=90]
  \GraphInit[vstyle=Hasse]
  \SetVertexLabel
  \SetUpEdge[style={->}]
  \SetGraphUnit{0.75}
  \Vertices{circle}{0,1,2,3,4,5,6,7,8,9,10}
  \foreach \v in {0,1,2,3,4,5,6,7,8,9,10}{
      \foreach \d in {1,3,4,5,9}{
          \pgfmathtruncatemacro{\w}{mod(round(\v+\d),11)}
          \Edges (\v, \w));
      }
  }
\end{tikzpicture}
\caption{\label{fig_T11} The tournament $ST_{11}$.}
\end{figure}
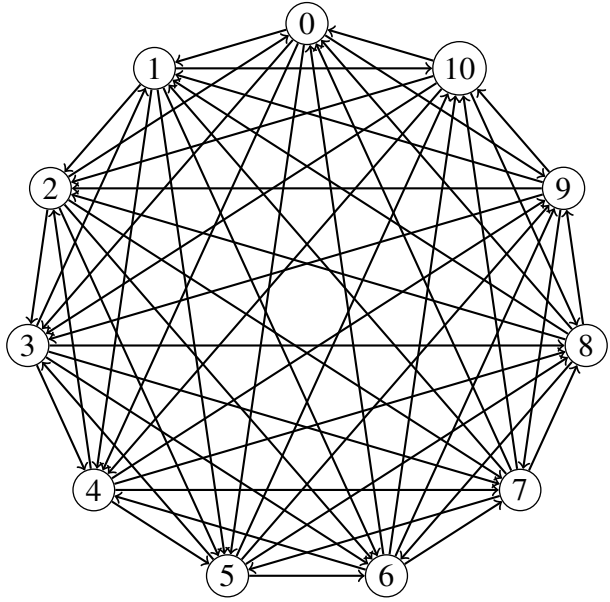

In~\cite{NL94} it is stated without proof that all oriented graphs on at most $16$ vertices are $4$-dicolourable. 
Here, we prove a weaker statement that suffices for our purposes.

\begin{proposition}\label{prop_n_leq15}
All oriented graphs on at most $15$ vertices are $4$-dicolourable. 
\end{proposition}
\begin{proof}
Let $\vec{G}$ be an oriented graph of order $n \leq 15$.

Assume first $n=15$.
According to Theorem~\ref{thm_tournament_14} \textit{(i)}, $\vec{G}$ contains an acyclic subdigraph $S$  of order
$5$. Now $\vec{G}-S$ has 10 vertices, and so is $3$-dicolourable by Theorem~\ref{thm_nl_clique}.
Thus  $\vec{G}$ is $4$-dicolourable.

Assume now that $13 \leq n \leq 14$. Then $\vec{G}$ contains an acyclic subdigraph $S$ of order $\lfloor \log_2 (13) \rfloor + 1=4$ by Lemma~\ref{lem_acyclique_log}. 
Now $\vec{G}-S$ has 10 vertices, and so is $3$-dicolourable by Theorem~\ref{thm_nl_clique}.
Thus  $\vec{G}$ is $4$-dicolourable.

Assume now that $n \leq 12$. Let $(A,B)$ be a partition of $\vec{G}$ with $|A|=|B|=6$.
By Proposition~\ref{prop_3crit_edges}, each of $\vec{G}\langle A\rangle$ and  $\vec{G}\langle A\rangle$ is $2$-dicolourable. Hence,
$\vec{G}$ is  $4$-dicolourable.
\end{proof}

\section{Dichromatic number of surfaces}\label{sec:results}

\subsection{General bounds} \label{sec:generalbound}

We first determine the asymptotic behaviour of the dichromatic number of a surface of given Euler characteristic. The upper bound was pointed out to us by Raphael Steiner.

\begin{theorem}\label{thm:bnd-gen}
There exist two constants $a_1$ and $a_2$ such that, for every surface $\Sigma$ with Euler characteristic $c\leq -2$, we have
  \[
   a_1\frac{\sqrt{-c}}{\log(-c)} 
  \leq \vec{\chi}(\Sigma) \leq
  a_2  \frac{\sqrt{-c}}{\log(-c)}
  \]
\end{theorem}

\begin{proof}

Let us first establish the lower bound.
By Proposition~\ref{prop:exist-tournament}, there exists a tournament $T$ of order $H(c)$ such that 
$\dic(T) \geq \frac{H(c)}{2\log H(c)+1}$.
But, by Theorem~\ref{thm:plongement-complet}, this tournament is embeddable on 
$\Sigma$.
So $\dic(\Sigma) \geq \dic(T) \geq \frac{H(c)}{2\log H(c)+1}$.
Since $H(c) = \left\lfloor \frac{7 + \sqrt{49-24c}}{2} \right\rfloor$, we get the lower bound.

\bigskip


To see the upper bound, let $\Sigma$ be a surface of Euler characteristic $c$ and dichromatic number $k\geq 4$. (We will choose the constant $a_2$ large enough to not care about smaller values of $k$.) Let $\vec{G}$ be a $k$-dicritical oriented graph with $n$ vertices and $m$ arcs embedded in $\Sigma$.

By Proposition~\ref{prop_5crit} and Theorem~\ref{thm:4crit} there is a constant $b_1$ such that $n\leq -b_1c$. Thus, Euler's Formula implies that there is a constant $b_2$ such that $m\leq -b_2c$. Now, by a result of~\cite{EGK91}, there is a constant $b_3$ such that $m\geq b_3k^2\log^2(k)$. This yields $b_4$ such that, $k\leq b_4  \frac{\sqrt{-c}}{\log(k)}$. Applying the lower bound $a_1\frac{\sqrt{-c}}{\log(-c)} 
  \leq k$, we get that there is a $b_5$ such that $\log(k)\geq b_5\log(-c)$. This yields the upper bound.
%
%
%
%
\end{proof}

\subsection{Projective plane, torus, Klein bottle, and Dyck's surface}

We begin with a lower bound.

\begin{theorem}\label{thm:N1}
If $\Sigma$ is a surface of Euler characteristic at most $1$, then
$\vec{\chi}(\Sigma) \geq 3$.
\end{theorem}
\begin{proof} 
We show $\dic(\mathbb{N}_1), \dic(\mathbb{N}_2) , \dic(\mathbb{S}_1) \geq 3$.
To prove the lower bound for $\mathbb N_1$, we construct an oriented graph embeddable on $\mathbb N_1$ with dichromatic number $3$. 

The complete graph on $6$ vertices $K_6$ can be embedded as a {\bf triangulation} of the projective plane, that is  is an embedding of $K_6$ in the projective plane such that all faces are triangles.
\begin{figure}[htp]
    \centering
    \includegraphics[width=.8\textwidth]{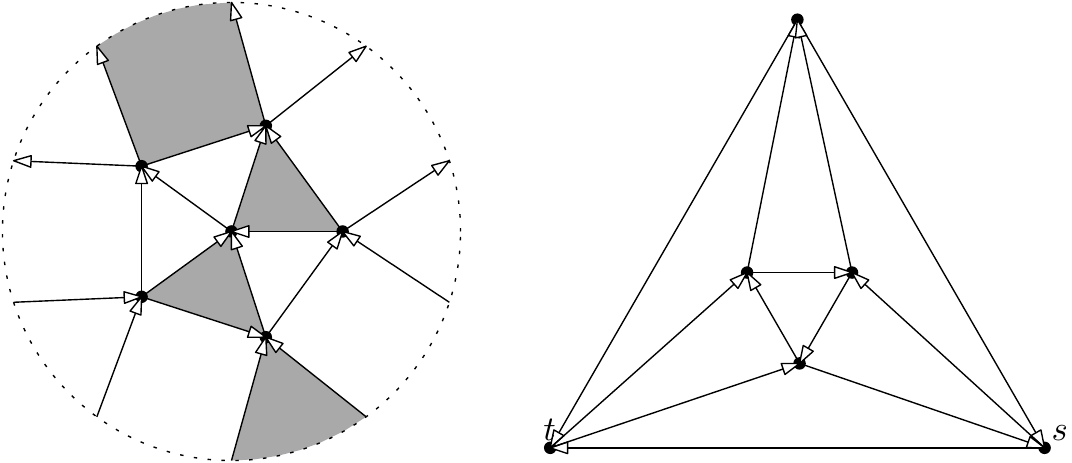}
    \caption{Left: an orientation $T$ of $K_6$ on the projective plane. Right: the gadget graph.}\label{fig:K6}
\end{figure}
Let $T$ be the orientation of $K_6$ displayed on the left of Figure~\ref{fig:K6}. Let $\vec{G}$ be the oriented graph obtained from $T$ by adding in each gray triangular face (which is a transitive tournament on three vertices with source $s$ and sink $t$), the gadget graph depicted on the left of   Figure~\ref{fig:K6}. Observe that in any $2$-dicolouring of the gadget  graph,
 the vertices of the outer face do not have all the same colour.

Assume now for a contradiction that $\vec{G}$ admits a $2$-dicolouring. Observe that either we have a monochromatic directed triangle in $T$ or one of the gray triangles is monochromatic.
But then the $2$-dicolouring cannot be extended to the gadget inside this transitive tournament by the above observation.
Hence $\vec{G}$ is not $2$-dicolourable.

While the above graph has 15 vertices, for the Klein bottle and the torus we also have orientations of $K_7\setminus e$ and $K_7$, respectively, of dichromatic number $3$, see Proposition~\ref{prop_3crit_edges} \textit{(ii)}.

Since any surface different from the sphere admits an embedding of one of the above graphs, we get the result.
\end{proof}

Let us continue with an upper bound.

\begin{theorem}\label{thm:N3}
$\vec{\chi}(\mathbb{N}_3) \leq 3$.
\end{theorem}

\begin{proof} 
Suppose for a contradiction that there exists a $4$-dicritical oriented graph $\vec{G}$ embeddable on $\mathbb{N}_3$. 
By Theorem~\ref{thm_nl_clique}, $12 \leq n(\vec G)$ (because $ST_{11}$ is an orientation of $K_{11}$ which is not embeddable on $\mathbb{N}_3$), and by Theorem~\ref{thm:4crit}, $n(\vec G) \leq 13$ (because $c(\mathbb N_3)= -1$), 

Let $T$ be the subdigraph induced by the vertices of degree $6$ (i.e. in-degree $3$ and out-degree $3$). Let $H = G-T$, and let $m(H,T)$ be the number of arcs with one end-vertex in $H$ and the other in $T$.
By Theorem~\ref{thm:gallai_or},  $T$ is a directed cactus. 
In particular, it has a vertex with at most two neighbours in $T$ and thus at least $4$ in $H$. Hence $n(H) \geq 4$. 

Set $\varepsilon(H) = \sum_{v \in V(H)}(d(v)-7)$.
Euler's Formula yields 
$6n(T) + 7 (n(\vec G)-n(T)) + \varepsilon(H)  =
2m(\vec{G}) \leq 6n + 6$ 
and so:
\begin{equation}\label{eq:n(T)}
    n(T) \geq n(\vec G) - 6 + \varepsilon(H) \geq n(\vec G)-6.
\end{equation}

By  Lemma~\ref{fait_t}, $m(T) \leq \frac{3}{2}n(T) - \frac{3}{2}$.
Hence 
\begin{align*}
m(H,T) = 6n(T) - 2m(T) \geq 3n(T) +3. 
\end{align*}

Since $T$ is a cactus, it is $2$-dicolourable, so $H$ must have dichromatic number at least $2$, and thus $m(H) \geq 3$. 
Now
\begin{eqnarray}
\varepsilon(H) & = &  m(H,T) + 2m(H)  - 7 n(H) \label{eq:eps}\\
 & \geq & 3n(T) + 9  - 7 n(H) \label{eq:eps2}
\end{eqnarray}

Assume for a contradiction that $n(H)=4$. Then $n(T) \geq 8$.
By \eqref{eq:eps2}, we have $\varepsilon(H) \geq 5$.
Thus by \eqref{eq:n(T)}, $n(T)\geq n(\vec G) -1$, a contradiction.
Hence $n(H) \in \{5,6\}$. 

All together, we have:
\begin{equation}\label{eq:all}
n(\vec G) \in \{12, 13\}, \quad n(T) \in \{6, 7, 8\} \quad \text{ and } \quad n(H) \in \{5, 6\}
\end{equation}

We are now going to bound the number of edges in $H$: \begin{align*}
m(H) & =   m(G) - m(H,T) - m(T)&\\
     & =  m(G) - \big(6n(T) - 2m(T)\big) - m(T)&\\
     & =  m(G) - 6n(T) + m(T)& \\
     & \leq 3n(\vec G) + 3 - 6n(T) + \frac{3}{2}(n(T) - 1)& \text{by Euler's formula and Lemma~\ref{fait_t}}\\
     & \leq 3n(\vec G) + 3 - \frac{9}{2}n(T) - \frac{3}{2}&\\
     & \leq 3n(\vec G) + 3 - \frac{9}{2}(n(\vec G) - 6) - \frac{3}{2}& \text{by (\ref{eq:n(T)})}\\
     &\leq 10& \text{by (\ref{eq:all})} 
\end{align*}

Let $A$ be a maximum acyclic subdigraph of $T$. 
By Lemma~\ref{lem:foret-cactus}, $n(A) \geq \lceil \frac{2}{3}n(T) \rceil$. 
So $n(T - A)\leq \frac{1}{3}n(T) \leq 2$ by (\ref{eq:all}), and since $n(H) \leq 6$, we have $n(\vec{G} - A) \leq 6+2=8$.  
Each vertex $v$ in $T - A$ must be in a cycle in $T\langle V(A) \cup \{v\}\rangle$, and so has at most $4$  neighbours in $\vec{G} - A$. Hence
$m(\vec{G} - A) \leq m(H) + 8 \leq 18$. 
Now, $\vec{G} - A$ is $2$-dicolourable by Proposition~\ref{prop_3crit_edges} and thus $\vec{G}$ is $3$-dicolourable. 
\end{proof}

Combining Theorems~\ref{thm:N1} and~\ref{thm:N3} with the fact that Dyck's surface is the torus plus a cross-cap determines the dichromatic number of the above surfaces.
\begin{corollary}\label{cor:lowgenus}
 $\dic(\mathbb{N}_1)=\dic(\mathbb{N}_2)=\dic(\mathbb{N}_3)=\dic(\mathbb{S}_1)=3$. 
\end{corollary}

\subsection{The dichromatic number of $\mathbb{S}_5$ and  $\mathbb{N}_{10}$}

By Theorem~\ref{thm:plongement-complet}, the complete graph on $11$ vertices is embeddable on every surface of Euler characteristic at most $-8$ and by Theorem~\ref{thm_nl_clique} its orientation $ST_{11}$ has dichromatic number $4$. Hence we have the following.

\begin{proposition}\label{prop:carac-8}
If $\Sigma$ is a surface of Euler characteristic at most $-8$, then
$\vec{\chi}(\Sigma) \geq 4$.
\end{proposition}

The remainder of the subsection is dedicated to the proof that the above inequality is tight for $\mathbb{S}_5$ and $\mathbb{N}_{10}$.

We shall need some preliminary notions and results.
Let $D$ be a digraph.
A {\bf list assignment} of $D$ is a mapping $L: V(D) \to \mathcal{P}(C)$, where $C$ is a set of colours.
An  {\bf  $L$-dicolouring} of $D$ is a dicolouring $\phi$ of $G$ such that $\phi(v) \in L(v)$ for all $v\in V(D)$.
If $D$ admits an  $L$-dicolouring, then it is {\bf  $L$-dicolourable}.


\begin{theorem}[Harutyunyan and Mohar~\cite{gallai_type_theorem}]\label{thm_gallai_list}
Let $D$ be a digraph and $L$ be a list assignment  of $D$ such that $|L(v)| \geq \max \{d^+(v), d^-(v)\}$ for every vertex $v\in V(D)$.
If $D$ is not $L$-dicolourable, then  $|L(v)| = \max\{d^+(v), d^-(v)\}$ for every vertex $v$ and every block of $D$ is either
\begin{itemize}
    \item a directed cycle, or
    \item a bidirected odd cycle, or 
    \item a bidirected complete graph.
\end{itemize}
\end{theorem}

%
%

The following result shows that in fact equality holds.

\begin{theorem}\label{thm:S5}
Every oriented graph embeddable on  $\mathbb{S}_5$ or $\mathbb{N}_{10}$ is $4$-dicolourable.
\end{theorem}

\begin{proof}
Let $\vec{G}$ be a $5$-dicritical oriented graph of order $n$ which is embedded in $\mathbb{S}_5$ or $\mathbb{N}_{10}$, and assume for a contradiction that
$\vec{G}$ is not $4$-dicolourable.

Let $T$ be the subdigraph induced by the vertices of degree $8$ (i.e. in-degree $4$ and out-degree $4$).
 Set $H=\vec{G}-T$,  $n_8 = n(T)$ and let $m(H,T)$ be the number of arcs with one end-vertex in $H$ and the other in $T$.
By Theorem~\ref{thm:gallai_or}, $T$ is a directed cactus and so  is $2$-dicolourable.
Therefore $H$ is not $2$-dicolourable. In particular, by Proposition~\ref{prop_3crit_edges}, $m(H) \geq 20$.

Euler's Formula yields 
$8n_8 + 9 (n-n_8) + \sum_{v \in V(H)}(d(v)-9) =
2m(\vec{G}) \leq 6n + 48$ 
and so:
\begin{equation}\label{minoration_n8}
    n_8 \geq 3(n-16) + \sum_{v \in V(H)} (d(v)-9) \geq 3(n-16)
\end{equation}

On the other hand, we have $\sum_{v \in V(T)} d(v) = 8 n_8 = 2m(T) + m(H,T) $ and
$m(\vec{G}) = m(H) + m(H,T) + m(T)$. We deduce
\begin{equation}\label{equa_nb_arcs}
    m(H) = m(\vec{G}) + m(T) - 8 n_8
\end{equation}

By Lemma~\ref{fait_t}, $m(T) \leq \frac{3}{2} (n_8-1)$. Thus
$20 \leq m(H) \leq m(\vec{G}) + \frac{3}{2}(n_8-1) - 8 n_8$. Hence $13n_8\leq 2m(\vec{G}) -43$.
With Eq.~\eqref{minoration_n8} and Euler's formula, it implies
\begin{equation}\label{encadrement}
3(n-16) \leq n_8 \leq \frac{2m(\vec{G}) - 43}{13} \leq \frac{6n+5}{13}
\end{equation}
After simplifying, we get $n \leq 19$. 
Moreover, by Proposition~\ref{prop_n_leq15}, we have $n \geq 16$.
We now distinguish few cases depending on the number $n$ of vertices.
\medskip

\noindent
\underline{Case $n=19$}:
By Eq. \eqref{encadrement}, we have $9\leq n_8 \leq \frac{119}{13}$ and so $n_8 = 9$.

Assume first that $m(T) = \frac{3}{2}(n_8 - 1) = 12$. By   Lemma~\ref{fait_t}, $T$ is connected and each block of $T$ is a directed triangle. So $T$ is Eulerian, i.e. $d^+_T(v)=d^-_T(v)$ for all $v\in V(T)$.

Since $n_8=9$, then $n(H) =10$. So, by Theorem~\ref{thm_nl_clique}, $H$ admits a $3$-dicolouring $\phi$
 with colour set $\{1,2,3\}$. Since all blocks of $T$ are directed triangles, $T$ contains a vertex $v$ such that $d^+_T(v)=d^-_T(v) =1$. 
So $v$ has $3$ out-neighbours in $H$. 
Let $v_1,v_2$ be two of these out-neighbours. 
Let us recolour $v_1$ and $v_2$ by setting $\phi(v_1)=\phi(v_2)=4$ (since there is no digon, the resulting colouring is still proper).
We then define for every vertex $x$ of $T$:
\[
L(x) = \{1,2,3,4\} \setminus \phi(N^+(x) \cap V(H))
\] 
Observe that an $L$-colouring of $T$ extends the $4$-colouring of $H$ into a $4$-colouring of $G$, so $T$ is not $L$-colourable. 
Observe that $|L(x)| \geq 4 -(4-d^+_T(x)) = \max\{ d^+_T(x),d^-_T(x)\}$ because $T$ is Eulerian. Moreover, since $v_1$ and $v_2$ are both coloured $4$, $|L(v)| \geq 2=  \max\{ d^+_T(x),d^-_T(x)\} +1$. So   $T$ is $L$-dicolourable by Theorem~\ref{thm_gallai_list}, a contradiction.  

\medskip
Therefore we have $m(T) \leq 11$. By Euler's Formula, $m(\vec{G})\leq 3 n + 24$, and by Eq.~\eqref{equa_nb_arcs} $m(H) = m(\vec{G}) - 8n_8 + m(T)$.
Hence $m(H) \leq 20$.
But $H$ is not $2$-dicolourable, so it contains a $3$-dicritical oriented subgraph $\Tilde{H}$, and  $m(\Tilde{H}) \leq 20$. 
By Proposition~\ref{prop_3crit_edges}, there is a unique such $3$-dicritical oriented graph and it has 7 vertices and 20 arcs.
Hence $n(\Tilde{H})=7$, $m(\Tilde{H}) =m(H) =20$ and $H$ is the disjoint union of $\Tilde{H}$ and a stable set $S'$ of size $3$.
 Observe that each vertex of $S'$ has degree at least $9$, which implies that they are adjacent to every vertex of $T$ and have degree exactly $9$. 

Now, $m(\Tilde{H}) < m(K_7)$, so there are two non-adjacent vertices $x,y$ in $\Tilde{H}$.
Thus $S=S'\cup \{x,y\}$ is a stable set of order $5$ in $H$.
Moreover, by Lemma~\ref{lem:foret-cactus}, $T$ has an acyclic subdigraph $A$ of order $6$.
Pick $v \in V(T) \setminus V(A)$. The subdigraph $B$ of $\vec{G}$ induced by $S \cup \{v\}$ is acyclic and has order $6$.
Let $G'= \vec{G} -  (A \cup B)$. Observe that $G'$ has order $19-6-6=7$. Recall that by Theorem~\ref{prop_crit_deg_digraphe}, oriented graphs on at most $6$ vertices are $2$-dicolourable. 

Let $w \in V(G')\cap V(T)$.
\begin{itemize}
    \item
    If $|N(w) \cap V(A)| \leq 1$, then the subdigraph $A'$ induced by $V(A) \cup \{w\}$ is acyclic.
 Hence $G$ can be partitioned into two acyclic subdigraphs $A'$ and $B$ and $G-A'\cup B$ which has order $6$ and so is $2$-dicolourable. Thus $\vec{G}$ is $4$-dicolourable, a contradiction.
    \item
    If $|N(w) \cap V(A) | \geq 2$, then as $w$ is adjacent to all vertices of $S'$, we have $d_{G'}(w) \leq 8 - 2 - 3 = 3$. 
Now, $G'-\{w\}$ is $2$-dicolourable, and since  $d_{G'}(w)=3$, $G'$ is also $2$-dicolourable, and thus $G$ is $4$-dicolourable, a contradiction. 
\end{itemize}

\bigskip

\noindent
\underline{Case $n=18$}:
By Eq. \eqref{encadrement}, we have $n_8 \geq 6$.
Let $u$ be a vertex of degree $8$ in $\vec{G}$ and consider
$\vec{G}' = \vec{G} - (N^+(u) \cup \{u\})$ which is of order $13$.
By Theorem~\ref{thm:euler-gen}, $\Ad(\vec{G}') \leq  6 + \frac{6\times8}{13} < 12 = \Ad(ST_{13})$. So $\vec{G}'\neq ST_{13}$, and so, by Theorem~\ref{thm_tournament_14}, $\vec{G}'$ has an acyclic subdigraph
$A_0$ of order $5$. Then the subdigraph $A$ of $\vec{G}$  induced by $V(A_0) \cup \{u\}$ is acyclic and has order $6$. 

Set $B = \vec{G} - A$. Then 
$m(B) = m(\vec{G}) - \sum_{v \in V(A)}d(v) + m(A) \leq 78-8\times 6 + m(A) \leq 30 + m(A) \leq 30 + \binom{6}{2} = 45$. 
Moreover, $B$ is not $3$-dicolourable, for otherwise $\vec{G}$ would be $4$-dicolourable.
Hence $B$ contains a $4$-dicritical subdigraph $\Tilde{B}$.
 Because $m(\Tilde{B}) \leq m(B) \leq 45 < 55 = m(ST_{11})$, the oriented graph $\Tilde{B}$ is not  $ST_{11}$.
Thus $\Tilde{B}$ has order $12$ by Theorem~\ref{thm_nl_clique}. 
Consequently, for every vertex $v$ of $B$, $B-v$ is $3$-dicolourable and the subdigraph induced by $V(A)\cup \{v\}$ is not acyclic for otherwise $\vec{G}$ would be $4$-dicolourable. 
Hence, for each $v \in B$, $v$ must have at least one in-neighbour and one out-neighbour in $A$ and therefore  $m(A,B) \geq 2 n(B) = 24$.

But then $m(B) = m(\vec{G}) - m(A,B) - m(A) \leq 54-m(A)$. Recall that $m(B) \leq 30+m(A)$. Thus 
\[
m(B) \leq \frac{1}{2} \bigg((30+m(A)) + (54 - m(A))\bigg) = 42.
\]

We now do a similar reasoning with $B$ as the one we just did with $\vec{G}$.  
Because $m(B) \leq 42 < 60$, by Corollary~\ref{cor_tournament_13_12}~(ii), $B$ has an acyclic subdigraph $A'$ of order $5$ .
Set $B'=B - A'$.
Then $B'$ is  not $2$-dicolourable for otherwise $B$ would be $3$-dicolourable.
Recall that $|\Tilde{B}| =|B|$ so $d^+_B (v), d^-_B (v) \geq 3$ for all $v\in V(B)$ by Proposition~\ref{prop_crit_deg_digraphe}.
Thus $m(B') = m(B) - \sum_{v \in V(A')} d_{B}(v) + m(A') \leq 42-6\times 5 + m(A') = 12+m(A')$.

Moreover $B'$ has order $7$. Thus, by Theorem~\ref{thm_nl_clique},
 $B'-v$ is $2$-dicolourable for all vertex $v$ of $B'$.
Therefore $|N(v) \cap V(A')| \geq 2$ for otherwise $B$ would be $3$-dicolourable. Hence $m(A',B') \geq 2 n(B') = 14$.
Consequently $m(B') = m(B) - m(A',B')-m(A) \leq 28-m(A')$.  Together with $m(B') \leq 12+m(A')$, this yields
\[
m(B') \leq \frac{1}{2} \bigg((12+m(A'))+(28-m(A'))\bigg) = 20.
\]
By Proposition~\ref{prop_3crit_edges}, $B'$ is uniquely determined and has exactly $20$ arcs.
Thus there are five vertices with degree $6$ in  $B'$, and two with degree $5$. 
Moreover, each vertex of $B'$ has at least two neighbours in $A'$ and two neighbours in $A$. Hence, five vertices of $B'$ have degree at least $10$ in $G$, and two have degree at least $9$ in $G$.  
Let us denote by $n_9$ and $n_{\geq 10}$ the number of vertices of degree $9$ and at least $10$, respectively.
We have
\[
2m(G) = 2 \times 78  = 8\times 11 + 9 \times 2 + 10 \times 5 \leq 8n_8+9n_9+10n_{\geq 10} \leq 2m(\vec{G}) 
= 2 \times 78
\]

We deduce that the degree list of the vertices of $\vec{G}$:
there are eleven vertices with degree $8$, two vertices with degree $9$ and five with degree $10$. But this contradicts Eq.~\eqref{encadrement} which states that $\vec{G}$ has at most eight vertices of degree $8$. 


\bigskip

\noindent
\underline{Case $n=17$}:
We have $n_8 \geq 3(17-16) = 3$.
Let $u$ be a vertex of degree $8$. Recall that $d^+(u)=d^-(u)=4$. So  $|N^+(u) \cup \{u\}|=5$.
Hence $\vec{G}' = \vec{G} - (N^+(u) \cup \{u\})$ has order $12$ and $m(\vec{G}') = m(\vec{G}) - \sum_{v \in N^+(u) \cup \{u\}} d(v) + m(N^+(u) \cup \{u\})$.
But  $m(\vec{G})\leq3n + 24 = 75$, $\sum_{v \in N^+(u) \cup \{u\}} d(v) \geq 5 \times 8$, and
 $m(N^+(u) \cup \{u\}) \leq \binom{5}{2}=10$.
Therefore $m(\vec{G}') \leq 75-40+10 = 45<60$.
Thus, by Corollary~\ref{cor_tournament_13_12}, $\vec{G}'$ has an acyclic
subdigraph of order $5$. Adding $u$ to this subdigraph, we obtain an acyclic
subdigraph $A$ of order $6$.
Set $B=\vec{G}-A$. 
Then $n(B)=11$ and  $m(B) = m(\vec{G}) - \sum_{v \in V(A)} d(v) + m(A) \leq 75 - 6 \times 8 + \binom{6}{2} =42$.
Hence $B \neq ST_{11}$ and so $B$ is $3$-dicolourable by Theorem~\ref{thm_nl_clique}~(ii). This implies that $\vec{G}$ is $4$-dicolourable, a contradiction.

\bigskip

\noindent
\underline{Case $n=16$}:
By Theorem~\ref{thm_tournament_14}, $\vec{G}$ has an acyclic subdigraph $A$ of order $5$. Set $B=\vec{G}-A$.
We have $n(B)=11$.

If $B$ is not $ST_{11}$, then by Theorem~\ref{thm_nl_clique}~(ii), it is $3$-dicolourable, and thus $\vec{G}$ is $4$-dicolourable, a contradiction.

Henceforth $B=ST_{11}$, so $m(B) = 55$.
We have $m(\vec{G}) \leq 3n+24= 72$. Thus $m(A,B) \leq m(A,B) + m(A) = m(\vec{G}) - m(B)
\leq 72 - 55 = 17$.
But $\frac{17}{11} < 2$, so there is a vertex $v$ of $B$ such that  $|N(v) \cap V(A)| \leq 1$. The subdigraph $A'$ induced $V(A)\cup\{v\}$ is then acyclic and of order $6$. The oriented graph $\vec{G}-A'$ has order $10$, so, by Theorem~\ref{thm_nl_clique}~(ii), it is $3$-dicolourable. Thus $\vec{G}$ is $4$-dicolourable. This contradiction completes the proof.
\end{proof}

Clearly, Theorem~\ref{thm:S5} also provides an upper bound for the dichromatic number of surfaces of higher Euler characteristic. Moreover,  Proposition~\ref{prop:carac-8} and Theorem~\ref{thm:S5} allow to determine the following dichromatic numbers precisely.

\begin{corollary}\label{cor:N10S5} $\dic(\mathbb{N}_{10})=\dic(\mathbb{S}_5)=4$.
\end{corollary}


\subsection{Dicritical digraphs embeddable in a fixed surface}

The goal of this section is to prove that for any surface $\Sigma$, there is a finite number of $(k+1)$-dicritical digraphs embeddable on $\Sigma$ for every $k \geq 6$, see Corollary~\ref{cor:finite-dicrit}.  

The following result bounds the number of $k$-dicritical digraphs embeddable on a surface, when $k \geq 8$.
\begin{proposition}\label{prop_8crit}
Let $k\geq 8$ and let $D$ be a $k$-dicritical digraph embedded in a surface with Euler characteristic $c$. Then
  $$
  n(D) \leq \frac{-6c}{k-7}
  $$
\end{proposition}

\begin{proof}
By Proposition~\ref{prop_crit_deg_digraphe}, $d^+(v),d^-(v) \geq k-1$ for every vertex $v$ of $D$.
Moreover, since there are at most two arcs between any two vertices, by Theorem~\ref{thm:euler-gen}, $\Ad(D) \leq 12 -  \frac{12c}{n(D)}$
and so $2(k-1) \leq 12 - \frac{12c}{n(D)}$.
Now $n(D) (k-1-6) \leq -6c$, and we obtain the result.
\end{proof}

One can however get better upper bounds following the method used by Gallai~\cite{Gal63a,Gal63b} for getting lower bounds on the density of critical graphs. This method is based on the concept of blocks. 
Recall that a graph $G$ is non-separable if it is connected and $G-v$ is connected for all $v\in V(G)$,
and that a block of $G$ is a subgraph which is non-separable and maximal with respect to this property. 
Let $A$ be the set of cut-vertices of $G$ and $\mathcal B$ the set of blocks of $G$. The \textbf{block forest} $B(G)$ of $G$ is the graph on vertices $A \cup \mathcal B$ where $aB$ is an edge of $B(G)$ if and only if $a \in A$, $B \in \mathcal B$ and $a \in B$. The block forest of a graph is a forest. If $G$ is connected, then $B(G)$ is also connected. It is then called the {\bf block tree} of $G$.
A {\bf leaf block} of a graph is a block which is a leaf in the block forest.
Such a  block has exactly one vertex in the union of all other blocks.
This vertex is the {\bf attachment} of the leaf block.
The blocks and the block forest of a digraph are simply those of its underlying multigraph.
A {\bf directed Gallai forest} is a digraph in which each block is a single arc, a directed cycle, a bidirected odd cycle, or a bidirected clique.

\begin{theorem}[Bang-Jensen et al..~\cite{hajosconstruct2019}]\label{thm:gallai}
Let $\vec{G}$ be a $k$-dicritical digraph.
The subdigraph induced by the vertices of in- and out-degree $k-1$  is a directed Gallai forest.
\end{theorem}

\begin{lemma}\label{lem:gallai-deg}
Let $k\ge 3$ be an integer. If $H$ is a directed Gallai forest of maximum degree at most $2k$ not containing $\overleftrightarrow{K}_{k+1}$, then 
\[ m(H) \leq \left(k-1+\frac{2}{k}\right) n(H).\]
\end{lemma}

\begin{proof}
We prove the result by induction. 

If $H$ is not connected, then we have the result by applying the induction hypothesis on each of its connected components, and summing the obtained inequalities.
If $H$ consists of a single block, then $H$ is either an arc, a directed cycle, a bidirected odd cycle, or a bidirected clique of order at most $k$. Hence $m(H) \leq (k-1)n(H)$, so we have the result.

Suppose now that $H$ is connected but not $2$-connected.

Assume moreover that $H$ contains a leaf block $H_1$ which is not $\overleftrightarrow{K}_k$. 
Then  
\[m(H_1) \leq \left(k-1+\frac{2}{k}\right)(n(H_1)-1).\]
Let $H_2$ be the union of the blocks distinct from $H_1$. By the induction hypothesis, we have
\[  m(H_2) \leq \left(k-1+\frac{2}{k}\right) n(H_2). \]
Because  $n(H) = n(H_1) + n(H_2)-1$ and $m(H) = m(H_1) + m(H_2)$, we get the result by summing the two above inequalities.

Henceforth, assume that every leaf block is $\overleftrightarrow{K}_k$. 
Let $L$ be a leaf block of $H$ which is the end of a diameter $\mathbb{D}$ in the block tree.
Let $P$ be the block incident to $L$ in $H$. It has maximum degree at most $2$ and thus it is a directed cycle or a single arc. In particular, $P$ is not a leaf block, and so
it is not an end of $\mathbb{D}$. Let $Q$ be the block distinct from $L$ which is incident to $P$ in $\mathbb{D}$, and let $L_1, \dots, L_q$ be the blocks incident to $P$ and distinct from $Q$. Since $\mathbb{D}$ is a diameter, each $L_i$ is a leaf block and thus a $\overleftrightarrow{K}_k$.  
In particular, it implies that $q\leq n(P) -1$. Set $a = n(P) - q -1$, and note that $a\geq 0$.
Let $H_1=P\cup \bigcup_{i=1}^q L_i$.
We have $m(H_1) = qk(k-1) + n(P)= qk(k-1) + q + 1 + a$ and $n(H_1) = q(k-1) + n(P)= qk +1  + a$. 

\begin{eqnarray*}
m(H_1) &  =  & qk(k-1) + q + 1 + a\\
& \leq & qk(k-1) +2q + \left(k-1+\frac{2}{k}\right) a  ~~~~~~\mbox{(because $q\geq 1$ and $a\geq 0$)}\\
& \leq & \left(k-1+\frac{2}{k}\right) qk + \left(k-1+\frac{2}{k}\right) a  = \left(k-1+\frac{2}{k}\right)(n(H_1)-1)
\end{eqnarray*}

Let $H_2$ be the union of the blocks which do not appear in $H_1$. By the induction hypothesis, we have
\[  m(H_2) \leq \left(k-1+\frac{2}{k}\right) n(H_2). \]
Because  $n(H) = n(H_1) + n(H_2)-1$ and $m(H) = m(H_1) + m(H_2)$, we get the result by summing the two above inequalities.
\end{proof}

\begin{theorem}\label{thm:ad-dicrit}
Let $k \geq 3$ be an integer. Let $D$ be a $(k+1)$-dicritical digraph different from $\overleftrightarrow{K}_{k+1}$. Then 
\[ m(D) \geq \left(k + \frac{k-2}{2k^2+3k-4}\right) n(D).\]
\end{theorem}

\begin{proof}
Let $S$ be the set of vertices $v\in V(D)$ such that $d^+(v)=d^-(v)=k$. By Theorem~\ref{thm:gallai}, the induced subdigraph $D\langle S\rangle$ is a directed Gallai forest, and by Lemma~\ref{lem:gallai-deg}, 
\[ m(D\langle S\rangle ) \leq \left(k-1+\frac{2}{k}\right)|S|.\]

Note that $2k|S|$ is the number of arcs of $D$ incident with vertices of
S, counting those in $D\langle S\rangle $ twice. Hence,
\begin{align}
m(D) \geq 2k|S| - m(D\langle S\rangle ) \ge \left(k+1-\frac{2}{k}\right)|S|. \label{eq:arcS}
\end{align}

All vertices in $V(D)\setminus S$ have degree at least $2k+1$, so
\begin{equation}
2m(D) \ge (2k+1)(n(D)-|S|) + 2k|S| = (2k+1) n(D) -|S|. \label{eq:arcG}
\end{equation}

Considering $(k+1-\tfrac{2}{k})$ Eq.\eqref{eq:arcG} + Eq.\eqref{eq:arcS}, we obtain
\begin{align*}
\left(2k+3-\frac{4}{k}\right)m(D) &\ge (2k+1)\left(k+1-\frac{2}{k}\right)n(D)\\
\frac{2k^2+3k-4}{k}m(D) &\ge \frac{2k^3+3k^2-3k-2}{k}\,n(D)\\
m(D) &\geq \frac{2k^3+3k^2-3k-2}{2k^2+3k-4}\,n(D) = \left(k + \frac{k-2}{2k^2+3k-4}\right) n(D).
\end{align*}
\end{proof}

\begin{corollary}\label{cor:finite-dicrit}
For any surface $\Sigma$ and any $k\geq 6$, there is a finite number of $(k+1)$-dicritical digraphs embeddable on $\Sigma$.
\end{corollary}
\begin{proof}
Let $\Sigma$ be a surface, and $k\geq 6$. 
Set $\epsilon_{k+1}= \frac{k-2}{2k^2+3k-4}$ and $c=c(\Sigma)$. 
Let $D$ be a $(k+1)$-dicritical digraph embeddable on $\Sigma$ distinct from $\overleftrightarrow{K}_{k+1}$.
Since there are at most two arcs between any vertices, by Theorem~\ref{thm:euler-gen}, we have $m(D)\leq 6n(D) - 6c$. Moreover,  by
 Theorem~\ref{thm:ad-dicrit}, $m(D)\geq \left(k + \epsilon_{k+1}\right) n(D)$.
 Hence $  (k+\epsilon_{k+1})n(D) \leq 6n(D) - 6c$ so 
 $$n(D)\leq \frac{-6c}{k-6+\epsilon_{k+1}}.$$   
\end{proof}


\subsection{Dicolouring planar digraphs and planar oriented graphs}


\begin{theorem}\label{thm:planar-2-NP}
Deciding whether a planar digraph is $2$-dicolourable is NP-complete.
\end{theorem}

\begin{proof}
We shall give a reduction from {\sc Planar 3-SAT} which consists in deciding whether a 3-SAT formula whose incidence graph\footnote{The \textit{incidence graph} of a 3-SAT formula is the bipartite graph with a vertex for each clause and each variable, and a variable is adjacent to a clause if it belongs to it.} is planar
is satisfiable. This problem was shown to be NP-complete by Lichtenstein~\cite{Lic82}.

A {\bf ${\bf \neq}$-gadget} between $u$ and $v$ is a digon between $u$ and $v$.
Trivially, a $\neq$-gadget is $2$-dicolourable and its extremities have distinct colours in any of its $2$-dicolourings.


Consider now an instance $\Phi$ of \textsc{PLANAR 3-SAT} 
and let $H$ be its incidence graph embedded in the plane.
Let us construct the planar oriented graph $\vec{G}$ from $H$ as follows.
First, we add a vertex $t_F$ in each face of $H$.
Now for every clause $C=\ell_x \vee \ell_y \vee \ell_z$, we replace the vertex $C$ and the three incident edges by a clause gadget as follows.
We replaced the vertex $C$ by a directed $6$-cycle $(x_C, u_C, y_C, v_C, z_C, w_C, x_C)$ inside which we add a vertex $t_C$ which is connected to $u_C$, $v_C$ and $w_C$ via $\neq$-gadgets. Let $F_1$ (resp. $F_2$, $F_3$) be the face containing $(x,C,y)$ (resp $(y,C,z)$, $(z,C,x)$) in $H$.
We add a $\neq$-gadget between $t_{F_1}$ and $u_C$, between $t_{F_2}$ and $v_C$, and between  $t_{F_3}$ and $w_C$.
Finally, for any $a\in \{x,y,z\}$, if $\ell_a$ is the negated literal $\bar{a}$, then add a $\neq$-gadget between the variable vertex $a$ and $a_C$, and if
 $\ell_a$ is the non-negated literal $a$, the  add a new vertex $\bar{a}_C$ and two $\neq$-gadgets between this vertex and $a$ and $a_C$.
See Figure~\ref{fig_gadget_b}.

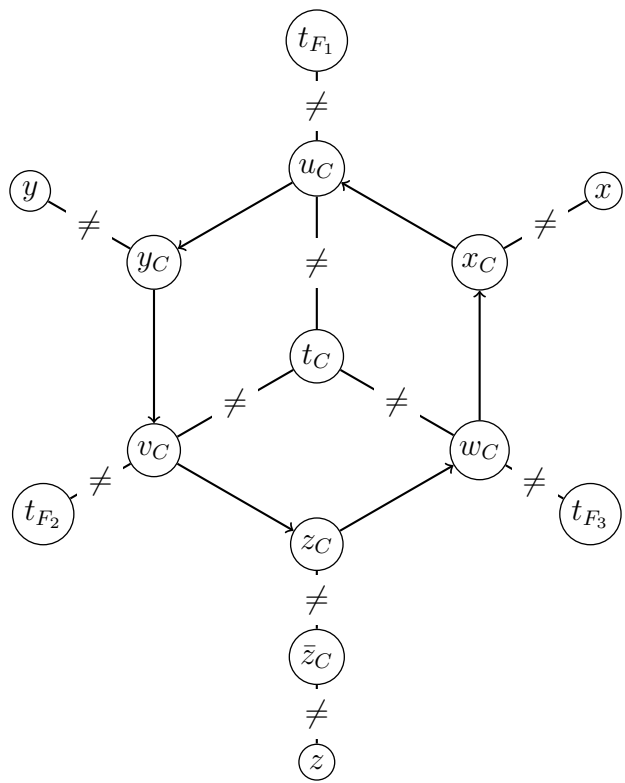
\begin{figure}[hbtp]
\centering
\begin{tikzpicture}[scale=2, rotate=30]
  \GraphInit[vstyle=Hasse]
  \SetVertexNoLabel
  \SetUpEdge[style={->}]
  \SetGraphUnit{0.75}
  \Vertex[a=0,   d=2.2, L=$x$, NoLabel=false]{x}
  \Vertex[a=120, d=2.2, L=$y$, NoLabel=false]{y}
  \Vertex[a=240, d=2, L=$\bar{z}_C$, NoLabel=false]{zb}
  \Vertex[a=240, d=2.7, L=$z$, NoLabel=false]{z}

  \Vertex[a=0,   d=1.25, L=$x_C$, NoLabel=false]{xx}
  \Vertex[a=120, d=1.25, L=$y_C$, NoLabel=false]{yy}
  \Vertex[a=240, d=1.25, L=$z_C$, NoLabel=false]{zz}
  
  \Vertex[a=60,  d=1.25, L=$u_C$, NoLabel=false]{g1}
  \Vertex[a=180, d=1.25, L=$v_C$, NoLabel=false]{g2}
  \Vertex[a=300, d=1.25, L=$w_C$, NoLabel=false]{g3}
  
  \Vertex[a=60,  d=2.1, L=$t_{F_1}$, NoLabel=false]{ag1}
  \Vertex[a=180, d=2.1,  L=$t_{F_2}$, NoLabel=false]{ag2}
 \Vertex[a=300, d=2.1,  L=$t_{F_3}$, NoLabel=false]{ag3}
  
  \Vertex[a=0, d=0, L=$t_C$, NoLabel=false]{g0}
  
  
  \Edge[label=$\neq$, style={-}](x)(xx);
  \Edge[label=$\neq$,  style={-}](y)(yy);
  \Edge[label=$\neq$,  style={-}](zb)(zz);
    \Edge[label=$\neq$,  style={-}](zb)(z);

  \Edges (xx,g1,yy,g2,zz,g3,xx)
  
  \Edge[label=$\neq$,  style={-}](g1)(g0);
  \Edge[label=$\neq$,  style={-}](g2)(g0);
  \Edge[label=$\neq$,  style={-}](g3)(g0);

 \Edge[label=$\neq$,  style={-}](g1)(ag1);
 \Edge[label=$\neq$,  style={-}](g2)(ag2);
 \Edge[label=$\neq$,  style={-}](g3)(ag3);

\end{tikzpicture}
\caption{\label{fig_gadget_b} Clause gadget associated to the clause $\neg x \vee \neg y \vee z$. }
\end{figure}

\medskip

Let us now show that $\Phi$ is satisfiable if and only if $\vec{G}$ is $2$-dicolourable.

\smallskip
Assume first that $\Phi$ is satisfiable.
Colour each variable with $1$ if it is true and $2$ if it is false. 
For each face $F$ of $H$, colour the vertex $t_F$ with $1$, and for each clause $C$ colour the vertex $t_C$ with $1$ and $u_C,v_C,w_C$ with $2$. 
It remains to colour $x_C$, $y_C$ and $z_C$ for each clause $C$. Each of these vertices is incident with a unique $\neq$-gadget, which forces to colour the vertex with the colour opposite to the one of the other end of the gadget.
Let us show that there is no monochromatic directed cycle.
Assume for a contradiction that there is such a cycle. Since two vertices linked by a $\neq$-gadget have distinct colours, such a cycle can only be 
one of the cycles $(x_C, u_C, y_C, v_C, z_C, w_C, x_C)$ for some clause $C=\ell_x\vee \ell_y\vee \ell_z$. But, as $\Phi$ is satisfied, at least one of the literals $\ell_x, \ell_y, \ell_z$ is true, and thus, by construction, at least one of the vertices $x_C, y_C, z_C$ is coloured $1$. But the vertices $u_C,v_C,w_C$ are coloured $2$, so the cycle is not monochromatic. Hence we have a $2$-dicolouring of $\vec{G}$.
\smallskip 

Assume now that $\vec{G}$ admits a $2$-dicolouring. Observe that the digraph induced by the $t_F$ for $F$ face~of~$H$, and all the $t_C, u_C, v_C, w_C$ for $C$ clause, and all the $\neq$-gadgets between them, is connected. Without loss of generality, we may assume  that all the $t_F$ and $t_C$ are coloured $1$ and all the $t_C, u_C, v_C, w_C$ are coloured $2$.
Let $\phi$ be the truth assignment defined by $\phi(x)=true$ if and only if $x$ is coloured $1$ in $\vec{G}$.
Consider a clause $C=\ell_x\vee \ell_y\vee \ell_z$. The directed cycle  $(x_C, u_C, y_C, v_C, z_C, w_C, x_C)$ is not monochromatic, so at least one vertex $x_C$, $y_C$ and $z_C$ is coloured $1$. By construction, this means that one of the literals $\ell_x, \ell_y, \ell_z$ is true.
Hence $\Phi$ is satisfied.
\end{proof}

Since every graph embeddable on the sphere can also be embedded in any other surface, Theorem~\ref{thm:planar-2-NP} implies that
{\sc $\Sigma$-2-Dicolourability} is NP-complete for any surface $\Sigma$.
It is then natural to ask about the complexity of the problem restricted to oriented graphs.

Recall that Conjecture~\ref{conj_neumann} states that every planar oriented graph is $2$-dicolourable.
If true, it implies that {\sc $\mathbb{S}_0$-Oriented-2-Dicolourability} can be trivially solved in polynomial time because the answer is always positive.
The following result shows that, conversely, if it happens to be false, {\sc $\mathbb{S}_0$-Oriented-2-Dicolourability} is NP-complete.

\begin{theorem}\label{thm:NP}
If Conjecture~\ref{conj_neumann} does not hold, then 
deciding whether a planar oriented graph is $2$-dicolourable is NP-complete.
\end{theorem}

\begin{proof}
The proof is similar to the one of Theorem~\ref{thm:planar-2-NP}.
The only difference is in the $\neq$-gadget, which should now be constructed without any digons.

Suppose that Conjecture~\ref{conj_neumann} does not hold. Then there is a planar $3$-dicritical oriented graph  $\vec{G}$.
Let $uv$ be an arc of $\vec{G}$. By definition of $3$-dicriticality, $\vec{G}\setminus uv$ is $2$-dicolourable.
Moreover, in any $2$-dicolouring $\phi$ of  $\vec{G}\setminus uv$, $u$ and $v$ are coloured the same for otherwise $\phi$ would be a $2$-dicolouring of $\vec{G}$.
We say that $\vec{G}\setminus uv$ is a {\bf =-gadget} between $u$ and $v$.

Let us now explain how to construct a {\bf $\neq$-gadget} between two vertices $u$ and $w$.
We start with four vertices $u$, $v_1$, $v_2$, $w$. We add the arcs $v_1v_2$, $v_2w$, $wv_1$ and two $=$-gadgets between $u$ and $v_1$ and between $u$ and $v_2$. See Figure~\ref{fig_gadget_neq}.
One easily sees that a $\neq$-gadget is $2$-dicolourable and that $u$ and $w$ have distinct colours in any of its $2$-dicolourings. 
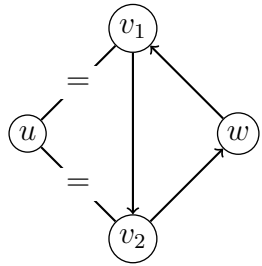
\begin{figure}[hbtp]
\centering
\begin{tikzpicture}[scale=1.4, rotate=0]
  \GraphInit[vstyle=Hasse]
  \SetVertexNoLabel
  \SetUpEdge[style={->}]
  \SetGraphUnit{0.5}
  
  \Vertex[a=0,   d=1, L=$w$, NoLabel=false]{w1}
  \Vertex[a=90,  d=1, L=$v_1$, NoLabel=false]{v1}
  \Vertex[a=270, d=1, L=$v_2$, NoLabel=false]{v2}
  
  \Vertex[a=180, d=1, L=$u$,   NoLabel=false]{v}
  
  \Edges (w1,v1,v2,w1)
  
  \Edge[label=$\eq$, style={-}](v)(v1);
  \Edge[label=$\eq$, style={-}](v)(v2);
\end{tikzpicture}
\caption{\label{fig_gadget_neq} A $\neq$-gadget between $u$ and $w$.}
\end{figure}
\end{proof}

\section{Concluding remarks}
We have determined the dichromatic number of the projective plane $\mathbb N_1$, the Klein bottle $\mathbb N_2$, the torus $\mathbb S_1$, Dyck's surface $\mathbb N_3$, the $5$-torus $\mathbb S_5$, and the $10$-cross surface $\mathbb{N}_{10}$. For the surfaces in between, the dichromatic number is either $3$ or $4$. We verified that all orientations of triangulations of the double torus $\mathbb S_2$ or $\mathbb{N}_{4}$ on at most $14$ vertices have dichromatic number at most $3$ by computer. Note that this does not imply that all $14$-vertex graphs embeddable on these surfaces have dichromatic number at most $3$, see e.g.~\cite{DP19}. By Theorem~\ref{thm:4crit}, it would suffice to check all digraphs on up to $22$ vertices.

\begin{problem}
 Determine the maximum dichromatic number of an oriented graph embeddable on the double torus $\mathbb S_2$ or in the $4$-cross surface $\mathbb{N}_{4}$. 
\end{problem}

A common generalization of colouring is list colouring. Similarly, dicolouring generalizes to  list dicolouring.
A digraph $D$ is {\bf  $k$-list-dicolourable} if it is $L$-dicolourable for every list assignment $L$ such that $|L(v)|\geq k$ for all $v\in V(D)$.
The {\bf  list dichromatic number} of a digraph $D$ is the least integer $k$ such that $D$ is $k$-list-dicolourable.
Note that by degeneracy, every planar oriented graph is  $3$-list-dicolourable. It has been asked whether the list version of Conjecture~\ref{conj_neumann} holds, see~\cite{BeArKh2018}. Combining degeneracy and Theorem~\ref{thm_gallai_list}, one can show that oriented graphs on the projective plane, the Klein bottle and the torus are $3$-list-dicolourable. However, our proofs of Theorems~\ref{thm:N3} and~\ref{thm:S5} do not generalize.

\begin{problem}
 Determine the maximum list dichromatic number of an oriented graph embeddable in Dyck's surface $\mathbb N_3$.
\end{problem}

We further believe that the asymptotic behaviour of the list dichromatic number of surfaces is an interesting topic of future research.
\medskip

Recall that Conjecture~\ref{conj_neumann} holds for digraphs of digirth $4$. Indeed, the proof of~\cite{LM17} gives a decomposition of planar triangulations into two induced chordal graphs. Any such decomposition of a graph is a $2$-dicolouring for all its orientations of digirth $4$. Unfortunately,   there are graphs embeddable in the projective plane that cannot be decomposed into two induced chordal graphs, see the left of Figure~\ref{fig:N13chordal}. On the other hand, we have verified that all orientations of triangulations in $\mathbb{N}_1$ with digirth $4$ and at most $17$ vertices have dichromatic number $2$. Indeed, the smallest $3$-dichromatic oriented graph of digirth $4$ that we know of contains a $K_{5,8}$, see the right of Figure~\ref{fig:N13chordal}. Hence, it is not embeddable in $\mathbb{N}_{10}$ nor $\mathbb{S}_{5}$, see~\cite{B78}. 

\begin{problem}
Determine the maximum dichromatic number of
 an oriented graph with digirth $4$ embeddable on the projective plane.
\end{problem}

\begin{figure}[htp]
    \centering
    \includegraphics[width=\textwidth]{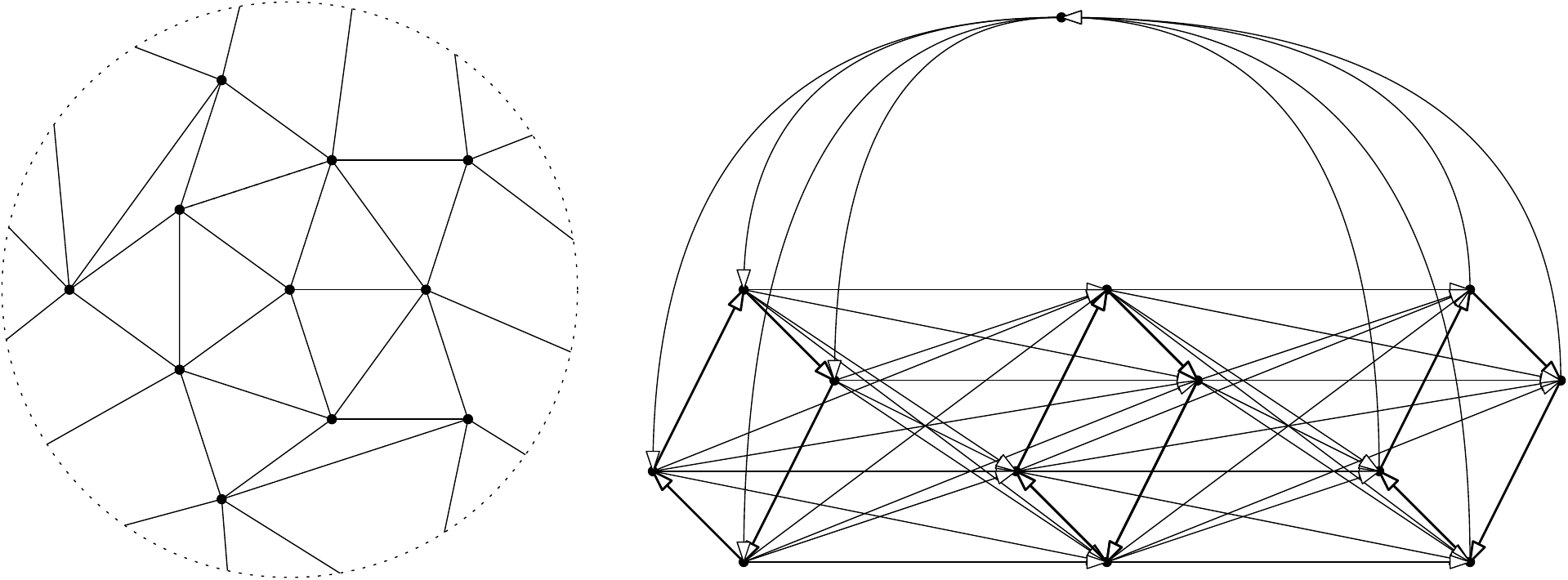}
    \caption{Left: A triangulation of $\mathbb{N}_1$ that cannot be decomposed into two chordal graphs. Right: A $3$-dichromatic digraph of digirth $4$.}\label{fig:N13chordal}
\end{figure}

On the other hand it follows from Theorem~\ref{thm:4crit}, that for a digraph $D$ on $\Sigma$ of Euler characteristic $c$, one has that digirth more than $4 - 9c$ implies $ \vec{\chi(D)} \leq 3$. We believe that the interplay of digirth, genus, and dichromatic number deserves further investigation.

\medskip

We prove  that  the number of $k$-dicritical digraphs embeddable on a surface is finite for any $k\geq 7$.
Thus, for $k\geq 6$ \textsc{$\Sigma$-$k$-Dicolourability} is polynomial time solvable. On the other hand we show that \textsc{$\Sigma$-$2$-Dicolourability} is NP-complete. Since $\chi(G) = \dic(\bid{G})$, the NP-completeness of the $3$-colourability of graph embedded in any fixed surface implies that
\textsc{$\Sigma$-$3$-Dicolourability} is NP-complete for all surface $\Sigma$. The following remains.

\begin{problem}\label{prob:45}
Let $\Sigma$ be a surface different from the sphere and $k\in \{4,5\}$. 
What is the complexity of \textsc{$\Sigma$-$k$-Dicolourability}~? Are there a infinitely many  $6$-dicritical digraphs embeddable on $\Sigma$ ?
\end{problem}

We further show that if Conjecture~\ref{conj_neumann} is false, then 
\textsc{$\mathbb{S}_0$-Oriented-$2$-Dicolourability}
is NP-complete. However, the method of this proof does not extend to prove the NP-completeness of \textsc{$\Sigma$-Oriented-$2$-Dicolourability} for a surface $\Sigma$ other than the sphere. Indeed, while assembling together planar gadgets results in a planar graph, assembling gadgets embeddable on a given surface does not necessarily results in a graph embeddable on this surface.

\begin{problem}
Let $\Sigma$ be a surface.
What is the complexity of \textsc{$\Sigma$-Oriented-$2$-Dicolourability}~?
\end{problem}

\section*{Acknowledgements}
We thank St\'ephane Bessy, Louis Esperet, Ararat Harutyunyan, Jocelyn Thiebaut, and Petru Valicov for fruitful discussions. We particularly thank Raphael Steiner for pointing us to the argument for the upper bound
in Theorem~\ref{thm:bnd-gen}. This research was supported by projects Digraphs: ANR-19-CE48-0013-01, GATO: ANR-16-CE40-0009-01, ALCOIN: PID2019-104844GB-I00, RYC-2017-22701 and ALGORIDAM: ANR-19-CE48-0016

\bibliographystyle{abbrv}
\bibliography{biblio}

\end{document}